\documentclass[12pt,a4paper]{article}
\usepackage{amssymb,amsmath,amsthm}
\usepackage{graphicx}
\usepackage{hyperref}

\newcommand\ex{\mathrm{ex}}
\newcommand\wex{\mathrm{wex}}
\newcommand\rex{\mathrm{rex}}

\theoremstyle{plain}
\newtheorem{theorem}{Theorem}[section]

\newtheorem{proposition}[theorem]{Proposition}

\newtheorem{problem}[theorem]{Problem}
\theoremstyle{definition}

\newcommand\cref[1]{Corollary~\ref{cor:#1}}

\title{Singular Tur\'an numbers and WORM-colorings}
\author{D\'aniel Gerbner$^1$ \and Bal\'azs Patk\'os$^1$ \and Zsolt Tuza$^{1,2}$ \and M\'at\'e Vizer$^1$ \\ \small $^1$ Alfr\'ed R\'enyi Institute of Mathematics\\ \small $^2$ Department of Computer Science and Systems
 Technology, University of Pannonia°
}
\date{}

\begin{document}

\maketitle

\begin{abstract}
    A subgraph $H$ of $G$ is \textit{singular} if the vertices of $H$ either have the same degree in $G$ or have pairwise distinct degrees in $G$. The largest number of edges of a graph on $n$ vertices that does not contain a singular copy of $H$ is denoted by $T_S(n,H)$. Caro and Tuza [Theory and Applications of Graphs, 6 (2019), 1--32] obtained the asymptotics of $T_S(n,H)$ for every graph $H$, but determined the exact value of this function only in the case $H=K_3$ and $n\equiv 2$ (mod 4). We determine $T_S(n,K_3)$ for all $n\equiv 0$ (mod 4) and $n\equiv 1$ (mod 4), and also $T_S(n,K_{r+1})$ for large enough $n$ that is divisible by $r$. 
    
    We also explore the connection to the so-called $H$-WORM colorings (colorings without rainbow or monochromatic copies of $H$) and obtain new results regarding the largest number of edges that a graph with an $H$-WORM coloring can have. 
\end{abstract}

\section{Introduction}

Tur\'an's paper \cite{T} about the maximum number of edges that a graph on $n$ vertices can have without containing a clique of size $k$ gave birth to extremal graph theory. The \textit{Tur\'an number} of a graph $G$, denoted by $\ex(n,G)$ is the maximum number of edges in an $n$-vertex $G$-free graph is a much studied and well understood parameter, if the chromatic number of $G$ is at least 3, but there are lots of open problems concerning the Tur\'an numbers of bipartite graphs $G$ (see the survey \cite{FS}). Tur\'an numbers were extended to hypergraphs and set systems (see Chapter 5 and Chapter 7 of \cite{GP}), and many variants are known. 

Motivated by the work of Albertson \cite{A92}, recently Caro and Tuza \cite{ct} introduced a new variant, the so-called \textit{singular Tur\'an number}. A copy of a graph $G$ in $H$ is called \textit{singular}, if the vertices $v_1,v_2,\dots, v_{|V(G)|}$ of the copy either have the same degree in $H$ or have pairwise different degrees in $H$. The singular Tur\'an number $T_S(n,G)$ is the maximum number of edges that a graph $H$ on $n$ vertices can have without containing a singular copy of $G$. Note that we have $\ex(n,G)\le T_S(n,G)$ for any graph $G$ and integer $n$.
 
Caro and Tuza  determined the asymptotics of $T_S(n,G)$ for every graph $G$. The Erd\H os-Stone-Simonovits theorem \cite{ES} states that, if $\chi(G)=p+1$, then  $\ex(n,G)=(1-\frac{1}{p}+o(1))\binom{n}{2}$. Caro and Tuza showed that if $|V(G)|=r+1$, then $$T_S(n,G)=(1-\frac{1}{pr}+o(1))\binom{n}{2}$$ (and clearly $T_S(n,G)=0$ if $G$ consists of a single edge). However, there was no exact result for any graph, except for the triangle in case $n=4k+2$, 
where the 4-partite Tur\'an graph is extremal for $T_S(n,K_3)$ as well. For other congruence classes modulo 4, they proved the following bounds.

\begin{theorem}[Caro, Tuza \cite{ct}] For\/ $k \ge 1$ we have:

\vspace{2mm}

\textbf{(i)} $6k^2-2\le T_S(4k,K_3)\le 6k^2-1$,

\vspace{1mm}

\textbf{(ii)} $6k^2+2k\le T_S(4k+1,K_3)\le 6k^2+3k-1$,

\vspace{1mm}

\textbf{(iii)} $6k^2+8k+1\le T_S(4k+3,K_3)\le 6k^2+9k+2$.

\end{theorem}

For the upper bounds, they used their general upper bound (which relies on the fact that a singular triangle-free graph is $K_5$-free in this case), and Tur\'an's theorem with the characterization of its extremal graphs (which shows that the $K_5$-free graphs with the largest number of edges contain singular triangles 
unless $n\equiv 2$ (mod 4)). 

For the lower bounds, they used the following constructions. If $n=4k$, consider a complete 4-partite graph with parts of size $k-1,k-1,k+1,k+1$. If $n=4k+1$, consider a complete 4-partite graph with parts each of size $k$, and join a new vertex to two of the classes. If $n=4k+3$, consider a complete 4-partite graph with parts of size $k,k,k+1,k+1$, and connect a new vertex to the $2k$ vertices in the two smaller parts.


Here we close the gap for two of the residue classes, and reduce it to 2 in the third case, by improving the upper bounds.
In the tight results the constructions above turn out to be extremal, and it is very likely that the situation is the same also for $n\equiv 3$ (mod 4).

\begin{theorem}\label{triangle} We have:

\vspace{2mm}

\textbf{(i)} $T_S(4k,K_3)=6k^2-2$ if\/ $k\ge 2$, and\/ $T_S(4,K_3)=5$,

\vspace{1mm}

\textbf{(ii)} $T_S(4k+1,K_3)=6k^2+2k$, and

\vspace{1mm}


\textbf{(iii)} $6k^2+8k+1\le T_S(4k+3,K_3)\le 6k^2+8k+3$.

\end{theorem}

We then apply some of our techniques to obtain better bounds on $T_S(n,K_{r+1})$. As Caro and Tuza observed, a graph without a singular copy of $K_{r+1}$ must be $K_{r^2+1}$-free, as otherwise  in a clique of size $r^2+1$ there are $r+1$ vertices either of the same degree or of pairwise distinct degrees. Tur\'an's result tells us that the graph with the largest number of edges without a $K_{r^2+1}$ is a balanced complete $r^2$-partite graph. We denote by $t(n,q)$ the number of edges in the balanced complete $q$-partite graph. Unless $r=2$ and $n=4k+2$, the balanced complete $r^2$-partite graph does contain singular copies of $K_{r+1}$. Moreover, it is not hard to see that there exist complete $r^2$-partite graphs without singular copies of $K_{r+1}$ if and only if $r$ divides $n$ and $n$ is at least $r^2(r+1)/2$.
In this case, we denote by $t'(n,r^2)$ the largest number of edges contained in such graphs. With this notation we have the following result.

\begin{theorem}\label{largeclique}
For any\/ $r\ge 3$ the following holds.

\vspace{1mm}

\textbf{(i)} If\/ $n$ is large enough and\/ $n$ is divisible by\/ $r$, then we have $$T_S(n,K_{r+1})=t'(n,r^2).$$ Moreover, any extremal graph is isomorphic to the unique complete\/ $r^2$-partite graph with\/ $r$ possible part sizes each appearing $r$ times such that the smallest and largest parts differ by at most\/ $r$.

\vspace{1mm}

\textbf{(ii)} If\/ $n=rk+m$ with\/ $1\le m<r$, then $$t(n,r^2)-m\frac{r-1}{r^2}n+C_r\le T_S(n,K_{r+1})\le t(n,r^2)-\frac{n}{r^2}+\sqrt{n}$$ for some absolute constant\/ $C_r$.
\end{theorem}

Even with the theorems above, there was no $F\neq K_2$ for which $T_S(n,F)$ was known for every $n$. Now we give an example for this by determining $T_S(n,P_3)$ for all cases.

\begin{proposition}\label{pthree}
\begin{displaymath}
T_S(n,P_3)=
\left\{ \begin{array}{l l}
2 & \textrm{if\/ $n=3$},\\
5 & \textrm{if\/ $n=4$},\\
\frac{n^2+2n}{4}-2 & \textrm{if\/ $n>4$ and\/ $n$ is divisible by\/ $4$},\\
\frac{n^2+2n-4}{4} & \textrm{if\/ $n$ is even, but not divisible by\/ $4$},\\
\frac{n^2+2n-3}{4} & \textrm{if\/ $n\neq 3$ is odd.}\\
\end{array}
\right.
\end{displaymath}
\end{proposition}

\bigskip

\subsection{$F$-WORM colorings}

Given graphs $F$ and $G$, an \textit{$F$-WORM coloring} of $G$ is an assignment of colors to the vertices of $G$ such that every copy of $F$ in $G$ has more than one, but fewer than $|V(F)|$ colors. In other words, there are neither monochromatic, nor rainbow copies of $F$ in the coloring of $G$ (WORM stands for `WithOut Rainbow or Monochromatic'). WORM coloring was introduced by Goddard, Wash and Xu \cite{gwx}. 

Most of the research regarding WORM colorings dealt with complexity issues, or the number of colors used. However, in the same paper \cite{gwx}, Goddard, Wash and Xu introduced $\wex(n,F)$, the largest number of edges in a graph on $n$ vertices that has an $F$-WORM coloring. They determined $\wex(n,P_3)$. 

Let us describe first how WORM colorings are related to singular graphs. Observe that if $G$ does not contain a singular $F$, then coloring the vertices of $G$ by their degrees, we obtain an  $F$-WORM-coloring. This implies $T_S(n,F)\le \wex(n,F)$. 
Also note that in the proof of the general upper bounds on $T_S(n,F)$ Caro and Tuza \cite{ct} do not use the special properties of singularity and the proof works for $\wex$ as well. Thus we can restate their upper bound in the following form.

\begin{theorem}[Caro, Tuza, \cite{ct}] 

\vspace{2mm}

\textbf{(i)} $\wex(n,K_{r+1})\le \ex(n,K_{r^2+1})$.

\vspace{1mm}

\textbf{(ii)} If\/ $F$ has $r+1\ge 3$ vertices and chromatic number\/ $p+1\ge 2$, then\/ $\wex(n,F)\le \ex(n,K_{pr+1})+o(n^2)$.

\end{theorem}

Observe that we have equality in \textbf{(i)}. Indeed, consider the $r^2$-partite Tur\'an graph, and color it with $r$ colors such that each color class consists of $r$ parts. Then each color class is $K_{r+1}$-free, and there is no rainbow $K_{r+1}$ as there are fewer than $r+1$ colors. Similarly in \textbf{(ii)} $\ex(n,K_{pr+1})$ is a lower bound, as we can color the $pr$-partite Tur\'an graph with $r$ colors (thus avoiding rainbow $F$) such that each color class is $p$-partite, thus $F$-free.

Having an asymptotic result does not leave much room for improvement in general, but we obtain a better result for every bipartite graph.

\begin{proposition}\label{worm}

If\/ $F$ is bipartite and has\/ $r+1\ge 3$ vertices, then\/ $\wex(n,F)\le \ex(n,K_{r+1})+\ex(n,F)$.

\end{proposition}

Note that as $F$ is bipartite, the quadratic term remains the same, but we replace the error term $o(n^2)$ with $\ex(n,F)$. Also, the proof remains valid for any graph, but if $\chi(F)\ge 3$, then the upper bound is useless as $\ex(n,K_{r+1})+\ex(n,F)$ is more than the number of edges in $K_n$.

\section{Singular Tur\'an numbers}

We will use the following theorem of Brouwer \cite{br}. 

\begin{theorem}\label{partite+}
If\/ $H$ is a\/ $K_{r+1}$-free graph on\/ $n$ vertices which is not\/ $r$-partite, then\/ $H$ has at most\/ $t(n,r)-\lfloor n/r\rfloor+1$ edges, assuming\/ $n\ge 2r+1$.\end{theorem}

Hanson and Toft \cite{ht} also 
characterized the extremal graphs (the same result was independently obtained in \cite{afgs,kp}).

The extremal graphs in the result of Hanson and Toft are somewhat similar to the constructions of Caro and Tuza described in the Introduction. One takes a complete $r$-partite graph, and adds a new vertex $v$, that is connected first to every vertex in all but two of the classes (in case of $r=4$, so far this construction is the same as the construction of Caro and Tuza). For the remaining two classes, one picks a vertex $u$ from one of them and a non-empty set $A$ of vertices from another. We assume that in both classes there remain at least one unpicked vertex, i.e.\ one of the classes has more than one, the other has more than $|A|$ vertices. Now one connects $v$ to $u$ and to the vertices of $A$, while one deletes the edges between $u$ and $A$. It is easy to see that this construction has indeed chromatic number more than $r$, but does not contain $K_{r+1}$. It does, however, contain a singular triangle in case $r=4$.

Let us mention that part \textbf{(i)} of Theorem \ref{triangle} could be deduced (with some simple case analysis for $k\le 2$) from the above result of Hanson and Toft,
but we give a self-contained proof. We restate  \textbf{(i)} of Theorem \ref{triangle} for convenience.

\begin{theorem} $T_S(4k,K_3)=6k^2-2$ if\/ $k\ge 2$, and\/ $T_S(4,K_3)=5$.

\end{theorem}



\begin{proof} We use induction on $k$, the statement is obvious for $k=1$. 

Let $k>1$ and assume the statement is valid for $k-1$. Let $G$ be graph on $4k$ vertices that does not contain a singular triangle. Recall that $G$ must be $K_5$-free. If $G$ is $K_4$-free, we are done by Tur\'an's theorem. Let $A$ be a set of 4 vertices that induces a $K_4$ and $B$ be the set of the remaining vertices. Every other vertex has at most 3 neighbors in $A$, otherwise they would form a $K_5$. Thus, there are at most $3(4k-4)$ edges between $A$ and $B$ and there are at most $6(k-1)^2$ edges inside $B$ by Tur\'an's theorem, as $G$ is $K_5$-free. This means we are done, unless at least $12k-13$ edges go between $A$ and $B$, i.e.\ all but (at most) one of the vertices in $B$ are connected to exactly three vertices of $A$. Let $a,b,c,d$ be the number of edges from vertices of $A$ to $B$, i.e.\ their degree minus three. By the above, $12k-13\le a+b+c+d\le 12k-12$. If three of the numbers $a,b,c,d$ are the same, or three are different, the corresponding vertices of $A$ form a singular triangle. Thus, say, $a=b\neq c=d$. Then $a+b+c+d$ is even, thus equal to $12k-12$. In this case every vertex of $B$ is incident to the same number (3) of edges that go outside $B$, hence the edges inside $B$ cannot contain a singular triangle. By induction, if $k>2$, there are at most $6(k-1)^2-2$ edges inside $B$ and we are done.

If $k=2$, we are left with the case, when the vertices of $A$ form a $K_4$, and there are 5 edges inside $B$, so they form form a $K_4$ minus an edge and there are 12 edges between $A$ and $B$, so in this case, say, $a=b=4$ and $c=d=2$. The vertex $v \in A$ corresponding to $a$ has degree of 7 and a vertex $w \in A$ corresponding to $c$ has degree of 5 and they are connected. Easy case analysis shows that either they have a common neighbor of degree 6, or there is a vertex in $B$ of degree 7. We have a singular triangle in both cases.
%
%
\end{proof}

We continue with \textbf{(ii)} of Theorem \ref{triangle}. We restate it here for convenience.

\begin{theorem}\label{tri2}
$T_S(4k+1,K_3)=6k^2+2k$.
\end{theorem} 

\begin{proof} The statement is trivial for $k=1$, thus we assume $k\ge 2$.
Consider a graph $G$ on $4k+1$ vertices without a singular triangle and recall that $G$ is $K_5$-free then. Assume first $\chi(G)\ge 5$. We can apply Theorem \ref{partite+}, obtaining that $G$ has at most $6k^2+3k-k+1$ edges. Moreover, $G$ cannot be the extremal graph in the construction of Hanson and Toft, thus $G$ has fewer than $6k^2+3k-k+1$ edges, which is the desired bound.

Assume now $G$ is 4-partite, and let $A,B,C$ and $D$ be the parts. If any of them is empty, $G$ is 3-partite, thus has at most $(4k+1)^2/3$ edges, finishing the proof as $k\ge 2$. Let $G'$ be the complement of $G$ with respect to this 4-partition, i.e.\ $uv\in E(G')$ if $u$ and $v$ are in different parts and $uv\not\in E(G)$. We claim that $|E(G')|\ge \min\{|A|,|B|,|C|,|D|\}$.

Assume first $|A|=|B|=|C|$. If there are vertices $a\in A$, $b\in B$ and $c\in C$ such that none of them is incident to an edge of $G'$, they all have the same degree in $G$ and they form a triangle in $G$, a contradiction. To avoid that, for one of the classes all the vertices in it have to be incident to an edge of $G'$, which proves the claim. In case $|A|<|B|<|C|$, the same argument works. The only remaining case is when two parts have the same size $|A|=|B|$ and the other two parts have the same size $|C|=|D|$, but that would mean an even number of vertices, a contradiction.

It is left to show that the complete 4-partite graph with classes $A,B,C,D$ has at most $6k^2+2k+\min\{|A|,|B|,|C|,|D|\}$ edges. Indeed, we can prove this by induction on $l=k-\min\{|A|,|B|,|C|,|D|\}$. This is trivial for $l=0$, and whenever $l$ increases, we can look at it as moving a vertex from the smallest class to another class. Each time we do that, the number of edges decreases by at least 1.
\end{proof}

The proof of \textbf{(iii)} of Theorem \ref{triangle} (which we restate below) goes similarly, so we only give a sketch.

\begin{theorem}
$T_S(4k+3,K_3)\le 6k^2+8k+3$.
\end{theorem} 

\begin{proof}[Sketch of proof] If $k=1$, $6k^2+8k+3=17$, and the upper bound by Caro and Tuza \cite{ct} is the same number. Thus, we can assume $k\ge 2$.
Consider a graph $G$ on $4k+3$ vertices without a singular triangle and recall that $G$ is $K_5$-free then. Assume first $\chi(G)\ge 5$. We can apply Theorem \ref{partite+}, obtaining that $G$ has at most $6k^2+9k+3-k+1$ edges. Moreover, $G$ cannot be the extremal graph in the construction of Hanson and Toft, thus $G$ has fewer than $6k^2+8k+4$ edges, 
finishing the proof.

Assume now $G$ is 4-partite, and let $A,B,C$ and $D$ be the parts. If any of them is empty, $G$ is 3-partite, thus has at most $\lfloor (4k+3)^2/3\rfloor$ edges, finishing the proof as $k\ge 2$. From here, the proof is exactly the same as the proof of Theorem \ref{tri2}.
\end{proof}

Ideas from the proofs above can be applied to obtain bounds on $T_S(n,K_{r+1})$ for larger values of $r$, too. Let us start with introducing some constructions. 

We distinguish two cases according whether $r$ divides $n$ or not. Suppose $n=rk$ and let $k=l_1+l_2+\dots +l_r$ with $l_i\neq l_j$ for any $1\le i< j\le r$. Then the complete $r^2$-partite graph $K_{s_1,s_2,\dots,s_{r^2}}$ with $s_{ir+1}=s_{ir+2}=\dots=s_{(i+1)r}=l_{i+1}$ for any $i=0,1,\dots,r-1$ does not contain any singular copy of $K_{r+1}$. Indeed, there are $r$ different degrees in $K_{s_1,s_2,\dots,s_{r^2}}$, and for any accessible degree $d$ there are $r$ parts such that the vertices of that part have degree $d$. We say that a complete $r^2$-partite graph has \textit{property $R$}, if there are $r$ possible sizes of parts, each appearing exactly $r$ times. Observe that the parameter $t'(n,r^2)$ defined in the Introduction is the same as the maximum number of edges in an $r^2$-partite graph on $n$ vertices satisfying property $R$. In particular, $t'(n,r^2)>0$ if and only if $r$ divides $n$ and $n\ge r^2(r+1)/2$.
Moreover, for these values of $n$ and $r$, it is quite simple to determine $t'(n,r^2)$. If there exist $i,j$ with $l_i<l_j-2$ such that none of $l_i+1$ and $l_j-1$ appear among the $l_h$'s, then replacing $l_i$ by $l_i+1$ and $l_j$ by $l_j-1$ increases the number of edges. This shows that if $l_1<l_2<\dots<l_r$ hold, then we have $l_r \le l_1+r$. Moreover, there is exactly one partition of $k$ into $l_1,l_2,\dots,l_r$ with this property. 
If $n$ is odd and $n\equiv 0$ (mod $r^2$), then the $l_i$'s are consecutive integers, while if $n\equiv ir$ (mod $r^2$), then $l_{r-i}+2=l_{r-i+1}$.
The situation is similar for $n$ even, but then the gap-free sequence corresponds to $n\equiv r^2\!/2$ (mod $r^2$).
Observe that whenever $t'(n,r^2)$ is defined, then $t(n,r^2)-t'(n,r^2)\le r^3$ holds.

Suppose next $n=rk+m$ with $1\le m \le r-1$. Then consider the complete $r^2$-partite graph $K_{s_1,s_2,\dots,s_{r^2}}$ on $rk$ vertices with property $R$, that has the largest number, i.e.\ $t'(rk,r^2)$ edges. Suppose $s_1<s_2<\dots <s_r$. Add $m$ new vertices
and join them to each other and each of them to all the vertices in parts of size $s_1,s_2,\dots,s_{r-1}$, to obtain $G_{n,r^2}$. We claim that $G_{n,r^2}$ does not contain any singular copy of $K_{r+1}$. Clearly, newly added vertices have lower degree than any of the old vertices. As we joined the new vertices to those old vertices that had one of the $r-1$ highest degrees, the $r-1$ highest degrees increased by $m$, the smallest degree remained the same, and we added a new degree. Therefore, there are $r+1$ different degrees in $G_{n,r^2}$.
Vertices whose degrees are all different cannot form a $K_{r+1}$ as newly added vertices are not joined to old vertices of the lowest degree. Vertices of the same degree cannot form a copy of $K_{r+1}$ either, as there are fewer than $r$ newly added vertices and the other degree classes remained the same. Observe that $t(n,r^2)-e(G_{n,r^2})\le m\frac{r-1}{r^2}n+C_r$ for some absolute constant $C_r$.

Let us show a better construction that works only if $n=rk+m$ with $1\le m\le r-2$. Let $n'=r(k+1)$ and let $G_{n'}$ be the complete $r^2$-partite graph $K_{s_1,s_2,\dots, s_{r^2}}$ with property $R$ having $t'(n',r^2)$ edges. Observe that property $R$ and $r\ge 3$ ensures that there exists at least one $s_i$ that is odd. Remove one vertex from $r-m$ partite sets $S_1,S_2,\dots,S_{r-m}$ of size $s_i$ to obtain $S'_1,S'_2,\dots, S'_{r-m}$. Then the degree of any vertex in unchanged partite sets decreases by $r-m$, while the degree of vertices in $\cup_{j=1}^{r-m}S'_j$ decreases by $r-m-1$. Observe that the size of $S'_j$ is even, therefore, as $r-m\ge 2$, there  exists a perfect matching in $G_{n'}[\cup_{j=1}^{r^2}S'_j]$. Let us remove this perfect matching to obtain $G_n$. Observe that for every vertex $v$ of $G_n$ we have $d_{G_{n'}}(v)-d_{G_n}(v)=r-m$ and thus $G_n$ admits $r$ degrees and the degree classes are $r$-partite. Therefore, $G_n$ does not contain a singular copy of $K_{r+1}$. Finally, observe that the number of edges in $G_n$ is $t(n,r)-\frac{r-m}{2r^2}n-C_r$.




\begin{proof}[Proof of Theorem \ref{largeclique}]
The lower bounds are given by the constructions above. To obtain the upper bounds, let us repeat the observation of Caro and Tuza: if a graph $G$ does not contain a singular copy of $K_{r+1}$, then it is $K_{r^2+1}$-free. Indeed, among the $r^2+1$ vertices of a $K_{r^2+1}$, either $r+1$ have the same degree or there are $r+1$ of them of pairwise distinct degrees. Suppose first that $\chi(G)>r^2+1$. Then Theorem \ref{partite+} implies that $e(G)\le t(n,r^2)-\frac{n}{r^2}+1$ holds, and we saw that all extremal graphs contain singular copies of $K_{r+1}$, so $e(G)\le t(n,r^2)-\frac{n}{r^2}$ must hold. Therefore we can assume that $G$ is a subgraph of a complete $r^2$-partite graph $K_{s_1,s_2,\dots,s_{r^2}}$ ($s_1 \le s_2  \le \dots  \le s_{r^2}$). If $K_{s_1,s_2,\dots,s_{r^2}}$ does not have property $R$, then $K_{s_1,s_2,\dots,s_{r^2}}$ contains either $r+1$ parts of pairwise different sizes or $r+1$ parts of the same size. In both cases, there must exist a part $U$ such that every $u\in U$ is adjacent to an edge in $K_{s_1,s_2,\dots,s_{r^2}}\setminus G$. Indeed, otherwise the untouched vertices would form a singular copy of $K_{r+1}$. This shows that $e(G)\le e(K_{s_1,s_2,\dots,s_{r^2}})-s_1$. 

If $s_1\ge \frac{n}{r^2}-\sqrt{n}$, then this implies $$e(G)\le t(n,r^2)-\frac{n}{r^2}+\sqrt{n}.$$ 

On the other hand if $s_1\le \frac{n}{r^2}-\sqrt{n}$, then $$e(G)\le e(K_{s_1,s_2,\dots,s_{r^2}})\le t(n,r^2)-n.$$ This finishes the proof of (ii) because if $r$ does not divide $n$, then there does not exist a complete $r^2$-partite graph with property $R$. The proof of (i) is also done as, by definition, complete $r^2$-partite graphs with property $R$ have at most $t'(n,r^2)$ edges.
\end{proof}

We finish this section with the proof of Proposition \ref{pthree}, which states the following.

\begin{displaymath}
T_S(n,P_3)=
\left\{ \begin{array}{l l}
2 & \textrm{if\/ $n=3$},\\
5 & \textrm{if $n=4$},\\
\frac{n^2+2n}{4}-2 & \textrm{if $n>4$ and $n$ is divisible by 4},\\
\frac{n^2+2n-4}{4} & \textrm{if $n$ is even, but not divisible by 4},\\
\frac{n^2+2n-3}{4} & \textrm{if $n$ is odd.}\\
\end{array}
\right.
\end{displaymath}

\begin{proof}[Proof of Proposition \ref{pthree}] The cases $n=3$ and $n=4$ are trivial.
For the other cases, as we have mentioned in the introduction, $T_S(n,F)\le \wex(n,F)$. Goddard, Wash and Xu \cite{gwx} showed 
\begin{displaymath}
\wex(n,P_3)=
\left\{ \begin{array}{l l}
\frac{n^2+2n}{4} & \textrm{if $n$ is divisible by 4},\\
\frac{n^2+2n-4}{4} & \textrm{if $n$ is even, but not divisible by 4},\\
\frac{n^2+2n-3}{4} & \textrm{if $n$ is odd.}\\
\end{array}
\right.
\end{displaymath}
The extremal constructions are $K_{\lfloor n/2\rfloor,\lceil n/2\rceil}$ supplemented with maximal matchings in both parts, which avoids singular $P_3$ in case $n$ is odd, hence extremal for $T_S$ as well.
In case $n$ is even, but not divisible by four, a further extremal graph for wex is $K_{n/2-1,n/2+1}$ with maximal matchings in both parts. This one avoids singular $P_3$. Thus we are done, except in the case $n=4k$. Our lower bound is given by the graph $K_{n/2-1,n/2+1}$ with maximal matchings in both parts.

To obtain the same upper bound, let $G$ be a singular $P_3$-free graph on $n>4$ vertices, and partition $E(G)$ into two parts: $E_1$ consists of the edges between vertices of the same degree, while $E_2$ consists of the edges between vertices of different degrees. By definition, $E_1$ is a matching, thus $|E_1|\le n/2$. If $E_2$ contained a triangle, those three vertices would have different degrees, thus we could find a singular $P_3$ among them. Therefore, $|E_2|\le n^2/4$. Note that if $|E_2|\le n^2/4-2$ or $|E_1|\le n/2-2$, then we are done, as $|E(G)|=|E_1|+|E_2|\le n^2/4+n/2-2$.

If  the graph with $E_2$ as its edge set has chromatic number at least 3, then we can use Theorem \ref{partite+} to obtain $|E_2|\le n^2/4-\lfloor n/2\rfloor+1\le n^2/4-2$. Thus $E_2$ defines a bipartite graph with parts $A$ and $B$. If $|A|\le n/2-2$, then $|E_2|\le n^2/4-4$. If $|A|=n/2-1$, then we are done, unless $E_2$ consists of all the edges between $A$ and $B$. In that case every edge of $E_1$ is inside $A$ or $B$, which have odd size, thus $E_1$ avoids two vertices, hence $|E_1|\le n/2-1$. This implies $|E(G)|=|E_1|+|E_2|\le n^2/4-1+n/2-1$.

Finally, if $|A|=n/2=|B|$, observe that we are done, if there are at least two edges between $A$ and $B$ that are not in $E_2$. Let $A'\subseteq A$ be the set of vertices connected to each vertex of $B$ with an edge in $E_2$, and similarly $B'\subseteq B$ be the set of vertices connected to each vertex of $A$ with an edge in $E_2$. Then $|A'|\ge n/2-1$ and $|B'|\ge n/2-1$, otherwise we are done. Also, the degrees of the vertices in $A'$ are different from those in $B'$, by the definition of $E_2$. But they are all incident to the same number of edges in $E_2$, thus the difference has to come from $E_1$. It means every vertex of, say $A'$ is incident to an edge of $E_1$ and no vertex of $B'$ is incident to an edge of $E_1$. But then $E_1$ avoids $|B'|\ge 3$ vertices, thus $|E_1|\le n/2-2$, finishing the proof.
\end{proof}

\section{WORM-colorings}

Let us start with the proof of Proposition \ref{worm}, which states that if a bipartite graph $F$ has $r+1\ge 3$ vertices, then $\wex(n,F)\le \ex(n,K_{r+1})+\ex(n,F)$.

\begin{proof}[Proof of Proposition \ref{worm}] Let us consider a graph $G$ on $n$ vertices with an $F$-WORM coloring. Let $G_1$ be the subgraph spanned by the edges that connect vertices of the same color and $G_2$ be the subgraph spanned by the edges that connect vertices of different colors. Then $G_1$ is $F$-free, thus has at most $\ex(n,F)$ edges.

Graph $G_2$ is not necessarily $F$-free, as it can contain a copy of $F$ with two nonadjacent vertices from the same color class, which is not rainbow. But if $G_2$ contained a copy of $K_{r+1}$, that would necessarily be rainbow, thus contain a rainbow copy of $F$, a contradiction. This shows $G_2$ has at most $\ex(n,K_{r+1})$ edges, finishing the proof.
\end{proof}

Observe that in the above proof, if the $F$-WORM coloring of $G$ has $t$ colors, then $G_1$ consists of $t$ vertex-disjoint graphs, thus has fewer than $\ex(n,F)$ edges if $t>1$ and $F$ is not a forest. On the other hand, if $t<p-1$, then $G_2$ has fewer than $\ex(n,K_p)$ edges. This shows that a careful analysis could improve the above bound.

In case $F$ is a forest, there is a chance the bound given in Proposition \ref{worm} is sharp. Let $T$ be a tree on $k+1$ vertices. Erd\H os and S\'os \cite{ErSo} conjectured that $\ex(n,T)\le (k-1)n/2$, with equality in case $k$ divides $n$, shown by the vertex-disjoint union of $n/k$ copies of $K_k$. This conjecture is known to hold for several classes of trees, including paths due to the Erd\H os-Gallai theorem \cite{ErGa}, and stars, trivially. 

\begin{proposition}
Let\/ $T$ be a tree on\/ $k+1$ vertices such that the Erd\H os-S\'os conjecture holds for\/ $T$. Let\/ $n$ be divisible by\/ $k^2$. Then\/ $\wex(n,T)=t(n,k)+(k-1)n/2$.
\end{proposition}

\begin{proof}
For the upper bound, observe that the properly colored edges do not contain $K_{k+1}$, while the monochromatic edges do not contain $T$.

For the lower bound, consider the balanced complete $k$-partite graph, let the colors correspond to the parts, and place $n/k^2$ copies of $K_k$ into every part.
\end{proof}

If $T=S_k$, the star with $k$ leaves and $k$ is odd, then the Erd\H os-S\'os conjecture holds with equality if $n$ is large enough, as shown by any $(k-1)$-regular graph. Therefore, we do not need the divisibility condition.

\begin{proposition}
Let\/ $k$ be odd and\/ $n$ large enough. Then\/ $\wex(n,S_k)=t(n,k)+(k-1)n/2$.
\end{proposition}

\begin{proof} The upper bound, again, follows from the fact that the properly colored edges do not contain $K_{k+1}$, while the monochromatic edges do not contain $S_k$. For the lower bound we take the Tur\'an graph, let the colors correspond to the parts, and place a $(k-1)$-regular graph into each part.
\end{proof}

Let us consider now a general construction. For a graph $F$ with $r+1$ vertices, consider the balanced complete $r$-partite graph $T(n,r)$ on $n$ vertices, and add into each part $A$ an $F$-free graph with $\ex(|A|,F)$ edges. Let $T(n,F)$ denote an arbitrary one of the graphs obtained this way. Then $T(n,F)$ admits an $F$-WORM coloring, namely the $r$-coloring according to the parts of $T(n,r)$.

Recall that Proposition \ref{worm} shows that for a graph $F$ with $r+1$ vertices $\wex(n,F)-\ex(n,K_{r+1})\le \ex(n,F)$. The next proposition shows that this difference is $\Theta(\ex(n,F))$.

\begin{proposition}\label{propp}
If\/ $F$ has\/ $r+1\ge 3$ vertices and chromatic number\/ $p+1\ge 2$, then\/ $\wex(n,F)=\ex(n,K_{r+1})+\Theta(\ex(n,F))$.
\end{proposition}

\begin{proof} The upper bound follows from Proposition \ref{worm}. The lower bound is given by $T(n,F)$, observing that it has at least $p\cdot  \ex(\lfloor n/p\rfloor,F)=\Theta(\ex(n,F))$ edges added to the original $\ex(n,K_{r+1})$.
\end{proof}

\section{Concluding remarks}

Let us return to the connection of singular Tur\'an problems and WORM colorings. The upper bound given in Proposition \ref{worm} for $\wex(n,F)$ immediately implies the same upper bound on $T_S(n,F)$, but the lower bound given by the construction $T(n,F)$ usually contains singular copies of $F$, as the degrees in different parts of $T(n,p)$ can be the same. Moreover, the additional $F$-free graphs may make the degrees different. 

The first problem we can deal with, the same way as earlier: instead of the balanced complete $r$-partite graph $T(n,r)$, we consider $T^*(n,r)$, which is a complete $r$-partite graph that is as balanced as possible, with respect to the condition that any two parts have different size. In $T^*(n,r)$ the degrees indeed give the coloring we want. However, this coloring can be ruined by the graphs we add inside the parts. To avoid this, we will add regular graphs. We still have to be careful, if we add graphs inside the parts with different regularities, then we have to avoid the final degrees coinciding in different parts. Still, if we only want to obtain a result similar to Proposition \ref{propp}, i.e.\ we are only interested in the order of magnitude, it is enough to add an $F$-free regular graph into the smallest part; then only the largest degrees increase.

This motivates us to initiate the study of \textit{regular Tur\'an problems}: what is the largest number $\rex(n,F)$ of edges in an $F$-free \emph{regular} graph on $n$ vertices? Note that it is equivalent to determining the largest regularity that an $F$-free graph on $n$ vertices can have.

Observe first that $\rex(n,F)$ is not monotone in $n$. For example $\rex(6,K_3)=9$ as shown by $K_{3,3}$, but $\rex(7,K_3)=7$, as a 4-regular triangle-free graph on 7 vertices would have more edges than the Tur\'an graph, there is no 3-regular graph on 7 vertices, and the 2-regular graph $C_7$ shows the equality. For larger $n$, asymptotically large flops also happen. If $n=2k$, we have $\rex(2k,K_3)=k^2$, as the Tur\'an graph $K_{k,k}$ is $k$-regular. If $n=2k+1$, a bipartite graph cannot be regular. A theorem of Andr\'asfai \cite{andr} states that if a triangle-free graph on $n$ vertices is not bipartite, its minimum degree is at most $2n/5$, showing $\rex(2k+1,K_3)\le (2k+1)^2/5$. 

We can show a quadratic lower bound for $\rex(n,F)$ for every nonbipartite graph $F$.

\begin{theorem}
If a nonbipartite graph\/ $F$ has odd girth\/ $g$, then
$$\rex(n,F)\ge n^2\!/(g+6)-O(n) .$$
\end{theorem}

\begin{proof}
The case of $n$ even is settled by $K_{n/2,n/2}$.
If $n$ is odd, let us write it in the form $n=(g+6)q+2r$ where $0\le r\le g+5$ (and of course $q$ is odd). We construct an $F$-free $(2q)$-regular graph with two connected components.
One component is a bipartite graph of order $4q+2r$, which is obtained from $K_{2q+r,2q+r}$, removing $r$ mutually edge-disjoint perfect matchings.
The other component is obtained from $C_{g+2}$, replacing each of the vertices by an independent set of size $q$ and each of the edges by a copy of $K_{q,q}$.
\end{proof}

\begin{problem}
For any non-bipartite graph\/ $F$, determine\/ $\liminf \rex(n,F)/n^2$.
\end{problem}





For any $F$ with chromatic number $p+1\ge 3$, we know that $\rex(n,F)=(1+o(1))\ex(n,F)$ for infinitely many values of $n$, namely for $n$ divisible by $p$, as the Tur\'an graph is regular in that case. If $F$ is a tree with $r+1$ vertices, the Erd\H os-S\'os conjecture states $\ex(n,F)\le (r-1)n/2$. If it holds, it also implies $\ex(n,F)=\rex(n,F)$ for $n$ divisible by $r$, as shown by the vertex-disjoint union of $n/r$ copies of $K_r$. In case $F$ is a forest, this construction might contain $F$, but it is easy to see that if $F\neq K_2$, then $\rex(2k,F)\ge k$.

However, if $F$ is bipartite and not a forest, we do not know how close $\ex(n,F)$ and $\rex(n,F)$ can be. In particular, we do not know if $\liminf \ex(n,F)/\rex(n,F)$ is bounded by a constant. What we do know is that there exists a $d$-regular graph which has girth at least $\ell+1$ (thus is $F$-free) on at most $2\frac{(d-1)^\ell-1}{d-2}$ vertices by a result of Erd\H os and Sachs \cite{ESA}. This shows that for infinitely many $n$ we have $\rex(n,C_l)\ge cn^{1+1/\ell}$ for some constant $c$. Hence $\limsup \ex(n,F)/\rex(n,F)=\infty$ holds if and only if $F$
is bipartite and
contains a cycle.


\section*{Acknowledgement} 
Research was supported by the National Research, Development and Innovation Office - NKFIH under the grants K 116769, KH130371 and SNN 129364. Research of Gerbner and Vizer was supported by the J\'anos Bolyai Research Fellowship. Research of Vizer was supported by the New National Excellence Program under the grant number \'UNKP-19-4-BME-287.


\begin{thebibliography}{99}

\bibitem{A92} M. Albertson. Tur\'an theorems with repeated degrees. \textit{Discrete Math.}, \textbf{100} (1992) 235--241.

\bibitem{afgs} K. Amin, J. Faudree, R. J. Gould, and E. Sidorowicz. On the non-$(p-1)$-partite
$K_p$-free graphs. \textit{Discuss. Math. Graph Theory}, \textbf{33} (2013) 9--23.

\bibitem{andr} B. Andr\'asfai. Graphentheoretische Extremalprobleme. \textit{Acta Math. Acad. Sci. Hungar.} \textbf{15} (1964) 413--418.

\bibitem{br}
A. Brouwer. Some lotto numbers from an extension of Tur\'an's theorem. \textit{Afdeling
Zuivere Wiskunde [Department of Pure Mathematics]}, \textbf{152} (1981).

\bibitem{ct}
Y. Caro and Zs. Tuza. Singular Ramsey and Tur\'an numbers. \textit{Theory and Applications of Graphs}, \textbf{6} (2019) 1--32.

\bibitem{ErSo}
P.~Erd\H{o}s.
 Extremal problems in graph theory. In:
 {\em Theory of Graphs and its Applications (Proc. Sympos. Smolenice,
  1963)}, 29--36, 1964.

\bibitem{ErGa} P. Erd\H os and  T. Gallai.  On maximal paths and
circuits of graphs. {\it Acta Math. Acad. Sci. Hungar.}, \textbf{10} (1959)
337--356.

\bibitem{ESA} P. Erd\H os and H. Sachs. Regul\"are Graphen gegebener Taillenweite mit minimaler Knoten-zahl. \textit{Wiss. Z. Martin-Luther-Univ. Halle-Wittenberg, Math.-Naturwiss.}, \textbf{12} (1963) 251--258.

\bibitem{ES}
P. Erd\H os and M. Simonovits. A limit theorem in graph theory. \textit{Studia Sci. Math. Hungar.}, \textbf{1} (1966) 51--57.

\bibitem{FS}
Z. F\"uredi and M. Simonovits. The history of degenerate (bipartite) extremal graph problems. In: \textit{Erd\H os Centennial}, Springer, Berlin, Heidelberg, 169--264, 2013.

\bibitem{GP}
D. Gerbner and B. Patk\'os. Extremal Finite Set Theory. CRC Press, 2018.

\bibitem{gwx} W. Goddard, K. Wash, and H. Xu. WORM colorings. \textit{Discuss. Math. Graph Theory}, \textbf{35} (2015) 571--584.

\bibitem{ht} D. Hanson and B. Toft. $k$-saturated graphs of chromatic number at least $k$. \textit{Ars. Combin.}, \textbf{31} (1991) 159--164.

\bibitem{kp} M. Kang and O. Pikhurko. Maximum $K_{r+1}$-free graphs which are not $r$-partite. \textit{Mat. Stud.}, \textbf{24} (2005)
12--20.

\bibitem{T}
P. Tur\'an. Egy gr\'afelm\'eleti sz\'els\H o\'ert\'ekfeladatr\'ol. \textit{Mat. Fiz. Lapok}, \textbf{48} (1941) 436--452.

\end{thebibliography}
\end{document}